\documentclass{amsart}
\usepackage{amsfonts}
\usepackage{amssymb}
\usepackage{amsmath}
\usepackage{lineno}
\newtheorem{theorem}{Theorem}[section]

\newtheorem{corollary}[theorem]{Corollary}

\newtheorem{examples}[theorem]{Examples}
\newtheorem{definition}[theorem]{Definition}
\newtheorem{definitions}[theorem]{Definitions}
\newtheorem{example}[theorem]{Example}

\newtheorem{lemma}[theorem]{Lemma}

\newtheorem{proposition}[theorem]{Proposition}
\newtheorem{remark}[theorem]{Remark}
\newtheorem{remarks}[theorem]{Remarks}

\newcommand{\NN}{\mathbb{N}}

\begin{document}

\title{Graphs of Commutatively closed sets}
\author{ Andr\'e Leroy and Mona Abdi }
\address{Andre Leroy\\
Laboratoire de Math\'ematiques de Lens, UR 2462\\
University of Artois\\
F-62300 Lens, France\\
Email:andre.leroy@univ-artois.fr
\\Mona Abdi\\ Department of Mathematics\\
Shahrood University of Technology\\
Shahrood, Iran
}

\begin{abstract}
The present work aims to exploit the interplay between the algebraic properties of rings
and the graph-theoretic structures of their associated graphs. We introduce commutatively closed graphs and investigate properties of commutatively closed subsets of a ring with the help of graph theory. A particular attention is paid to constructions of subsets of a given diameter. In particular, we compute the diameter of artinian semisimple  algebras. 
\end{abstract}

\maketitle

\section*{Introduction}
 
A subset $S$ of a ring $R$ is {\it commutatively closed} if for any $a,b\in R$ such that $ab\in S$, we also have $ba\in S$.  This notion was introduced and studied in \cite{AL}.  The aim of this paper is to investigate further this property and, in particular, to involve graph theory in the subject.   Let us first mention some easy examples.   Any subset of a commutative ring is commutatively closed.  The set $\{1\}$ is commutatively closed if and only if the ring $R$ is Dedekind finite.  Similarly, the set $\{0\}$ is commutatively closed if and only if the ring $R$ is reversible.  The intersection of two commutatively closed subsets is easily seen to be commutatively closed and hence every subset of $R$ is contained in a minimal commutatively closed subset: its commutative closure.  The {\it closure} of a subset  $S\subset R$ is denoted by $\overline{S}$.   In \cite{AL} different characterizations of $\overline{S}$ were given. It was shown that this notion is related to various kind of subsets such as idempotents elements,  nilpotent elements, invertible elements, and various kinds of regular elements.  It was also used to characterize, semicommutative rings, $2$-primal rings, and clean rings . 
 We will briefly recall some of the notions used in \cite{AL} at the end of this introduction. 
\par In section 1, we describe a graph structure on the sets $\overline{\{a\}}$ and define the diameter. Our aim is to compute the diameter of some classical rings, in particular, the ring of matrices over a division ring. In section 2, we give some ways of constructing $\overline{\{a\}}$  and study some particular rings such as the free algebras. In the last section, we investigate some properties of the commutatively closed graph of matrix rings and semisimple algebras.
\par For two elements $a,b\in R$ we write $a\sim_1 b$ if there exists $c,d\in R$ such that $a=cd$ and $b=dc$.  We then define by induction $a\sim_{n+1} b$ if and only if there exists an element $c\in R$ such that $a\sim_1 c$ and $c\sim_n b$. For two elements $a,b\in R$, we also define $a\sim b$ if there exists $n\in \mathbb{N}$ such that $a\sim_n b$. For a subset $S\subseteq R$
 we denote $S_i$ the set $\{x\in R \mid x\sim_i s, {\rm \; for\;  some \;} s \in S \}$.  Since our ring $R$ is unital, the chain $S_i$ is ascending and we have  $\overline{S}=\cup_{i\ge 0}S_i$  (for details about these constructions we refer the reader to \cite{AL}). 
\par We recall that an element $a\in R$ is called {\it von Neumann regular} if there is $x\in R$ such that $a=axa$. Similarly, we define $a\in R$ to be a {\it $\pi$-regular} element of $R$ if $a^{n}xa^{n}=a^{n}$ for some $x\in R$ and $n\geq 1$. An element $a\in R$ is called {\it right (left) $\pi$-regular}, if $a^{n+1}x=a^{n}$ ($xa^{n+1}=a^{n}$) for some $x\in R$ and $n\geq 1$.   We call $a\in R$ {\it strongly $\pi$-regular} if it is both left and right $\pi$-regular.
\par Throughout this paper $R$ will be a unital ring, $U(R)$ and $N(R)$ will stand for the set of invertible and nilpotent elements of $R$, respectively. For an element $a$ of a ring $R$ we denote $l(a)$ (resp. $r(a)$) its left (resp. right) annihilator.
\par
Let us mention some results from \cite{AL}.  They give motivation for the subject and basic information about the results.
 
\begin{theorem}\cite[Theorem 2.7]{AL}
\label{0.1}
Let $\varphi : R \longrightarrow S$ be a ring homomorphism, then
\begin{enumerate}
\item For any $X\subseteq R$, $\varphi(\overline{X})\subseteq \overline{\varphi(X)}$.
\item If $\varphi$ is a ring isomorphism, then for any $X\subseteq R$, $\varphi(\overline{X})=\overline{\varphi(X)}$.
\item If $T\subseteq S$ is commutatively closed in $S$, then $\varphi^{-1}(T)$ is commutatively closed in $R$.
\item If $S$ is reversible, $Ker(\varphi)$ is commutatively closed.
\item If $S$ is Dedekind-finite, then $\varphi^{-1}(\{1\})$ is commutatively closed.
\end{enumerate}
\end{theorem}
\begin{proposition}
\cite[Proposition 3.6]{AL}
\label{0.2}
\begin{enumerate}
\item If $R$ is Dedekind-finite and $a\in U(R)$, then $\overline{\{a\}}=\{uau^{-1} ~|~ u \in U(R)\}$.
\item The set of unit $U(R)$ of a ring $R$ is commutatively closed if and only if $R$ is Dedekind-finite.
\end{enumerate}
\end{proposition}
\begin{proposition}
\cite[Proposition 4.7]{AL}
\label{0.3}
Let $k$ be a commutative field and $n\in \mathbb{N}$, the class of $\overline{\{0\}}$ in $M_{n}(k)$ is the set of nilpotent matrices.
\end{proposition}

For notions related to graph theory, we refer the reader to \cite{W}.

%There is an example of a Dedekind finite ring $R$ such that $M_n(R)$ is not Dedekind finite (See \cite{Lam}).
%-------------------------------

\vspace{5mm}

\section{Commutatively closed graphs and their diameters}

First, we start this section with some examples of commutatively closed sets.\par
%--------------------------

\begin{examples}
{\rm
	\begin{enumerate}
		\item A ring $R$ is reversible if and only $\{0\}$ is commutatively closed;
	    \item A ring $R$ is Dedekind finite if and only if $\{1\}$ is commutatively closed;
	   % \item Let $R$ be a reversible ring then $diam_R\big(C(0)\big)=0$
	   \item For a right $R$-module $M_R$ the set  $E= \{u\in R \mid ann_M(u-1)\ne 0 \}$ is commutatively closed.   Indeed, if 
	   $u=ab\in E$ and $0\ne m\in M$ is such that $m(ab-1)=0$, then $ma\ne 0$ and $ma(ba-1)=m(ab-1)a=0$.  This gives that $ba\in E$.
	   \item The ring $R$ is symmetric if and only if for every $a,b,c\in R$, $abc=0$ implies that $acb=0$. This can be translated into asking that for any $a\in R$, $r(a):=\{b\in R \mid ab=0\}$ is 
	   commutatively closed.  We refer the reader to \cite{G} for more information on this kind of ring.
	   \item $U(R)-1$ is always commutatively closed.   
	   More generally, denoting the set of regular elements by 
	   $Reg(R)=\{r\in R \mid   
	   \exists x \in R \;{\rm such \; that}\; r=rxr\}$ we have that 
	   $Reg(R)-1$ is always commutatively closed.  A similar result is true for the set of unit regular elements and also for the set of 
	   strongly $\pi$-regular elements.  
	   For proofs of these facts we refer the reader to  \cite{LN}.
	   \item We consider $Z_l(R)=\{a\in R \mid r(a)\ne 0\}$ the set of left zero divisors. 
	   We define similarly $Z_r(R)$ the set of right zero divisors.
	   Similarly as in the item above, $Z_r(R)+1$ (resp. $Z_l(R)+1$) is commutatively closed (\cite{AL}).
	   \item The set $N(R)$ of nilpotent elements of a ring $R$ is easily seen to be commutatively closed.  
	   \item Recall that an element $r\in R$ is strongly clean if there exist an invertible element $u\in U(R)$ and an idempotent $e^2=e\in R$ such that $eu=ue$ and $r=e+u$.  
	   It is proved in Theorem 2.5 of \cite{G} that the set of strongly clean elements is commutatively closed.  In the same paper the author also shows that the set of Drazin (resp. almost, pseudo) invertible elements is also commutatively closed.
	\end{enumerate}
}
\end{examples}
%-----------------
The three first statements of the following results were obtained in \cite{AL}.  So we will only prove the last one.

\begin{proposition}
\label{First basic prpoerties of connected elements}
	\begin{enumerate}
		\item For any $n\ge 1$ and $a,b\in R$, we  have
		$a\sim_n b$ if and only if there exist two sequences of elements in $R$ $x_1,x_2,\dots,x_n$ and $y_1,y_2,\dots, y_n$ such that $a=x_1y_1, y_1x_1=x_2y_2, y_2x_2=x_3y_3,
		\dots , y_nx_n=b$.
		\item If $a\sim_n b$, then $a-b$ is a sum of 
		$n$ additive commutators.
		\item If $a\sim_n b$, then there exist $x,y\in R$ such that $ax=xb$ and $ya=by$.  Moreover, for $l\in \mathbb{N}$, we also have $b^{n+l}=ya^lx$ and $a^{n+l}=xb^ly$.  In particular, $b^n=yx$ and $a^n=xy$.  
		%\item ***what happens if $S$ is commutative? 
		\item If elements $a,b$ in a ring $R$ are such that $a\sim_m b$, then for any $0<l<m$, we have $a^l\sim_q b^l$, where $q$ is such that $q= \left \lceil{\frac{m}{l}}\right \rceil$ 
	\end{enumerate}
\end{proposition}
\begin{proof}
	(4) We can write $m=lq-r$ for some $0\le r <l$ and, since $1_R\in R$, $a\sim_m b$ implies that we also have $a\sim_{m+r=lq}b$. 
	We prove that $a^l\sim_q b^l$ by induction on $q$.  
	If $q=1$, the statement (3) above shows that $a^m=xy$ and $b^m=yx$ for some $x,y\in R$ and hence $a^m\sim_1 b^m$.  
	
	So let us now assume that $q>1$ and put $n:=m+r=lq$, so that we have $a\sim_n b$.  We use the same notations as in statement ($1$) and assume that there exist two sequences of elements in $R$ $x_1,x_2,\dots,x_n$ and $y_1,y_2,\dots, y_n$ such that $a=x_1y_1, y_1x_1=x_2y_2, y_2x_2=x_3y_3,
	\dots , y_nx_n=b$.   We than have that $a=x_1y_1 \sim_l y_lx_l\sim_{l(q-1)} b$.  The case $q=1$ and the induction hypothesis lead to $a^l\sim_1 (y_lx_l)^l\sim_{q-1}b^l$.  This gives the conclusion.  
\end{proof}
%------------------------
%\begin{remark}
%The above proposition could suggest that, for an arbitrary element $a\in R$, we have $\{a\}_{m}^{l}=\{a^{l}\}_{m}$?''
%Answer of this question is ''no''. Consider $A=\left(
%\begin{array}{ccc}
%0 & 1 & 1  \\
%0 & 0 & 1 \\
%0 & 0 & 0
%\end{array} \right)$. Clearly  $A\in \{0\}_{2}$, but $A\notin \{0\}_{2}^{2}$.
%\end{remark}
%********Need to work more on this or to be more explicit.  I can see easily that $A\in \{0\}_{2}$but why  $A\notin \{0\}_{2}^{2}$.  We have 
%$\{0\}_{2}^{2}=\{xy \mid x,y\in \{0\}\}$ right ?*****

\vspace{2mm}

\begin{remarks}
	{\rm
	\begin{enumerate}
		\item 	With the same notations as in the above proposition \ref{First basic prpoerties of connected elements}, if $a\sim_n b$ and $x,y$ are such that $ax=xb$ and $by=ya$, then left multiplication by the elements $y$ and $x$ give rise to maps in $Hom_R(R/aR,R/bR)$
		and $Hom_R(R/bR,R/aR)$.   We then have the compositions $L_x\circ L_y=L_{a^n}$ and $L_y\circ L_x=L_{b^n}$.  In particular, if $a\sim b$ then $Hom_R(R/aR,R/bR)\ne \{0\}$.
		\item In section 3, we will study the ring of matrices over a division ring $D$.  It is interesting to observe that if $A,B \in M_n(D)$ are such that $A\sim B$, then the central spectrum $Sp(A)$ and $Sp(B)$ are equal. Indeed, there exists $m \in \mathbb{N}$ such that $A \sim_{m} B$. Proposition \ref{First basic prpoerties of connected elements} (3) shows that there exist matrices $X,Y\in M_n(D)$ such that $AX=XB$ and $YA=BY$. So, if $Av=\lambda v$ for some $0\ne v\in D^n$ and $\lambda \in k$ ($k$ is the center of $D$), then $BYv=YAv=Y\lambda v=\lambda Yv$.  Noting that $\lambda^m v=A^mv=XYv$, we conclude that $Yv\ne 0$ and hence $\lambda$ is indeed a central eigenvalue of $B$.  Let us remark that in the case $D$ is commutative, the equality between the spectrums is a consequence of the fact that the characteristic polynomials of $A$ and $B$ are the same (cf. \cite{AL}).
	\end{enumerate}
}
\end{remarks}

%--------------------------
In the following, we define the commutatively closed graph and show that the graph $C(a)$ is always connected for every $a\in R$. Also, we analyze some other of his properties.

%--------------------------
\begin{definitions}
\begin{enumerate}
\item[(1)] Let $a$ be an element in a ring $R$ and let $C(a)$ denote its commutative closure as defined in the introduction.   We define a graph structure with the elements of $C(a)=\overline{\{a\}}$ as vertices and two distinct vertices $x$ and $y$ of
$C(a)$  are said to be adjacent if and only if $y \in \{x\}_{1}$.   The commutatively closed graph of a ring $R$ is the union of all the graphs $C(a)$, for $a\in R$.  It will be denoted by $C(R)$.
\item[(2)]
Let $a$ be an element in $R$. In the class $\overline{\{a\}}$ of $a$, we define a distance as follows: For two elements $x, y \in \overline{\{a\}}$, we put $d(x,y)=min\{n \in \mathbb{N}~ |~ y \sim_{n} x\}$.
\\ One can easily check that $d$ is indeed a distance defined on $\overline{\{a\}}$.
\item[(3)] 
Let $R$ be a ring and $a\in R$, the diameter of a graph $C(a)$ is defined as follows:
\begin{center}
$diam\big(C(a)\big)=sup\big\{ d(x, y)~|~ x, y \in \overline{\{a\}} \big\}$.
\end{center}
Also, we define the diameter of a set $S\subseteq R$ of a ring $S$ as follows:
\begin{center}$diam(S)=sup \big\{ diam\big(C(a)\big) ~| ~ a \in S \big\}$. \end{center}
\end{enumerate}
\end{definitions}
%----------------------------
\begin{theorem}
\label{The graph C(a)}
\begin{enumerate}
	\item for $a,b\in R$, we have $a\sim b$ if and only if $b\in \overline{\{a\}}$.
	\item The relation $\sim$ on $R$ is an equivalence relation.
	\item If $b\in C(a)$ then, for any $l\in \NN$, $b^l\in C(a^l)$.
	\item For $a\in R$, the graph $C(a)$ is connected.
	\item A subset $S$ of $R$ is closed and connected if and only if it is the closure of an element of $R$.
	\item If $\overline{S}$ is the closure of a subset $S\subseteq R$ then $diam(S)=diam(\overline{S})$.
\end{enumerate}

\end{theorem}
\begin{proof}
	We leave the easy proof to the reader.	
%The result is an easy consequence of Proposition \ref{First basic prpoerties of connected elements}.
\end{proof}
\begin{proposition}
\label{3}
Let $a$ be an element in a ring $R$.  If $n\in \mathbb{N}$ is the smallest integer such that $\overline{a}=\{a\}_{n}$, then  $n\leq diam_{R}\big(C(a)\big)\leq 2n$.
\end{proposition}
\begin{proof}
We know that the distance between $a$ and every other element of $\{a\}_n$ is at most $n$. Now, if $b,c\in C(a)$ then $d(b,c)\le d(b,a)+d(a,c)\le 2n$.  So that $ diam_{R}\big(C(a)\big)\le 2n$.  The fact that $n$ is minimal such that $\overline{a}=\{a\}_{n}$ implies that $n\le diam_{R}\big(C(a)\big)$.
%Thus the diameter of $C(a)$ is at least $n$, since $\overline{a}=\{a\}_{n}$. The greatest distance between the elements $b,c \in {a}_n$ occurs when only path between $b$ and $c$ as follows: $b=b_{1}b_{2} \sim_{1} b_{2}b_{1} \sim_{1} b_{4}b_{3} \sim_{1} \dots \sim_{1} b_{n-1}b_{n}=a=c_{n-1}c_{n} \sim_{1} \dots \sim_{1} c_{2}c_{1} \sim_{1} c_{1}c_{2}=c$. Thus $d(b, c)=2n$. 
%Therefore $diam_{R}\big(C(a)\big)\leq 2n$, as desired.
\end{proof}
%------------------------
\begin{lemma}
\label{4}
Let $R$ and $S$ be two rings, and $(a,b) \in R\times S$. Then $\overline{\{(a,b)\}}=\overline{\{a\}} \times \overline{\{b\}}$.
\end{lemma}
\begin{proof}
Let $(c,d) \in \overline{\{(a,b)\}}$. Thus $(c,d) \sim_{n} (a,b)$, for some $n \geq 0$. So there are $(c_{1},d_{1}), \dots, (c_{n},d_{n}) \in R\times S$ such that $(c, d)\sim_{1} (c_{1}, d_{1})\sim_{1} \dots \sim_{1} (c_{n}, d_{n})=(a,b)$. One can easily see that $c \sim_{n} a$ and $d \sim_{n} b$. Hence $(c, d) \in \overline{\{a\}} \times \overline{\{b\}}$.
\par Now, let $(c, d) \in \overline{\{a\}} \times \overline{\{b\}}$. Since $c \in \overline{\{a\}}$ and $d\in \overline{\{b\}}$, thus there are $m, n \in \mathbb{N}$ such that $c \sim_{m} a$ and $d \sim_{n} b$. So we have $c \sim_{1} c_{1} \sim_{1} \dots \sim_{1} c_{m}=a$ and $d\sim_{1} d_{1} \sim_{1} \dots \sim_{1} d_{n}=b$, for some $c_{1}, \dots, c_{m} \in R$ and $d_{1}, \dots, d_{n} \in S$. We may assume that $m< n$. Thus we have $(c,d) \sim_{1} (c_{1},d_{1})\sim_{1} \dots \sim_{1} (c_{m}, d_{m})=(a, d_{m})\sim_{1} (a, d_{m+1}) \sim_{1} \dots \sim_{1} (a, d_{n})=(a,b)$. Hence $(c,d) \sim_{n} (a,b)$, and so $(c, d) \in \overline{\{(a,b)\}}$.   
\end{proof}
%%%%%%%%%%%%%%%%%%%
\begin{proposition}
\label{diameter of a product}
Let $R, S$ be two rings, and let $diam(R)=d_{1}$ and $diam(S)=d_{2}$. Then $diam(R\times S)=max\{d_{1}, d_{2}\}$.
\end{proposition}
\begin{proof}
Assume that $(a,b)\in R\times S$. For every $(a_{1}, b_{1}), (a_{2},b_{2}) \in \overline{\{(a,b)\}}$, we have $a_1, a_2 \in \overline{\{a\}}$ and $b_1, b_2 \in \overline{\{b\}}$, by Lemma \ref{4}.
Since $diam(R)=d_{1}$ and $diam(S)=d_{2}$, then $ d(a_{1}, a_{2})\leq d_{1}$ and $ d(b_{1}, b_{2})\leq d_{2}$. Thus there exist $t\leq d_1$ and $s\leq d_2$ such that  $a_1 \sim_{t} a_2$ and $b_1 \sim_{s} b_2$. Then we have $a_1 \sim_{1} c_{1} \sim_{1} \dots \sim_{1} c_{t}=a_2$ and $b_1\sim_{1} v_{1} \sim_{1} \dots \sim_{1} v_{s}=b_2$, for some $c_{1}, \dots, c_{t} \in R$ and $v_{1}, \dots, v_{s} \in S$. Let $t< s$. We have $(a_1, b_1) \sim_{1} (c_{1},v_{1})\sim_{1} \dots \sim_{1} (c_{t}, v_{t})=(a_2, v_{t})\sim_{1} (a_2, v_{t+1}) \sim_{1} \dots \sim_{1} (a_2, v_{s})=(a_2,b_2)$. Hence $d\big((a_{1}, b_{1}), (a_{2}, b_{2})\big) \leq max\{d_{1}, d_{2}\}$. Therefore $diam(R\times S)=max\{d_{1}, d_{2}\}$.
\end{proof}
%------------------------
\begin{remark}
{\rm
In the above proposition, we showed that if $R$ and $S$ are two rings with finite diameter, then the diameter of $R\times S$ is also finite. 

Let us remark that is easy to construct elements $a,b,c,d$ in 	a ring such that $a\sim_{1}b$ and $c\sim_{1}d$, but there is no path between $ac$ and $bd$. For instance, consider the free algebra $K<X,Y>$ where $K$ is a field.
Let $a=XY,b=YX, c=d=X^2Y$.  We leave to the reader to check that there is no path from  $ac=XYX^2Y$ to $bd=YX^3Y$.
%The way is to show that there is no other factorization of  
% $ac=X(1+YZ)Y$ if  $ac=X(1+YZ)Y=fg$ then X divides f on the left %and Y divides g on the right.  so we are left with $1+YZ= f'g'$ %but $1+YZ$ is an atom... 
}
\end{remark}
%--------------------------
\begin{lemma}
\label{6}
If $R$ is a Dedekind finite ring and $a \in U(R)$, then $diam_{R}(C(a))=1$.
\end{lemma}
\begin{proof}
Since $R$ is Dedekind finite and $a \in U(R)$, thus $\overline{\{a\}}=\{uau^{-1} ~|~ u \in U(R) \}$, by Proposition \ref{0.2}. So for every $b, c \in \overline{\{a\}}$, there exist $u, v \in U(R)$ such that $b=uau^{-1}$ and $c=vav^{-1}$. Hence $b=uv^{-1}cvu^{-1}$. Thus $b$ and $c$ are adjacent. Therefore $diam_{R}(C(a))=1$.
\end{proof}
\begin{proposition}
1) $R$ is commutative if and only  if $diam(R)=0$.
	
\noindent 2) Let $R$ be a division ring. Then $diam(R)=1$.

\end{proposition}
\begin{proof}
1) The first statement is a direct consequence of the definition.

\noindent 2) The second statement is easily obtained from Lemma \ref{6}.
\end{proof}

%Notice that if $R$ is a reduced ring. Then $diam_{R}\big(C(0)\big)= 0 = diam_{R[x]}\big(C(0)\big)$.
%Since $R$ is reduced, thus $R[x]$ and $R$ are reversible, and so $\{0\}$ is commutatively closed in $R$ and $R[x]$. 

%\begin{proposition}
%If $diam(R)=1$, then $R$ be Dedekind-finite.
%\end{proposition}
%\begin{proof}
%Assume that $diam(R)=1$. Thus $diam\big(C(a)\big)\leq1$ for every $a\in R$. So $diam\big(C(1)\big)=0$ or $1$. If $diam\big(C(1)\big)=0$, then the result follows. Let $diam\big(C(1)\big)=1$. Then there exist $a, b\in R$ such that $ab=1$, but $ba\neq 1$. So we have $e_{ij}=b^{i}(1-ba)a^{j}$ which is a non zero element of $R$ for every $i, j \in \mathbb{N}$. Consider the element $e_{12}+e_{23}$. Clearly $e_{12}+e_{23}$ is an element of $\overline{\{0\}}$. We show that $e_{12}+e_{23}\notin \{0\}_{1}$. Assume that $e_{12}+e_{23}\in \{0\}_{1}$. Thus there exist elements $x, y \in R$ such that $e_{12}+e_{23}=xy$ implies that $yx=0$. Then $(e_{12}+e_{23})^{2}=0$, and so $e_{13}=0$. Since $e_{13}=b(1-ba)a^{3}$, $b(1-ba)a^{3}=0$. By multiplying both sides of the equation in $a$ and $b^{3}$ from left and right respectively, we conclude $ba=1$, which is a contradiction. Hence $e_{12}+e_{23}\in \overline{\{0\}}\setminus \{0\}_{1}$, and so $diam\big(C(0)\big)\geq2$. This contradicts $diam(R)=1$. Therefore $R$ is Dedekind-finite.
%\end{proof}
%There exist Dedekind-finite rings $R$ such that $diam(R)\neq 1$(see Example \ref{aa}). Thus the converse of above proposition is not true.
%%%%%%%%%%%%%%%%%
\begin{theorem}
If $R$ is not Dedekind-finite, then $diam(R)=\infty$. 
\end{theorem}
\begin{proof}
Assume that $R$ is not Dedekind-finite. Thus there exist $a, b \in R$ such that $ab=1$ but $ba\neq  1$. So we have nonzero element $e_{ij}=b^{i}(1-ba)a^{j}$. Consider $A=e_{12}+e_{23}+\dots+e_{n-1,n}$. One can easily check that $A^{n}=0$ and $A^{n-1}=e_{1n}\neq 0$, since $ba\neq 1$. 
Write $A=(e_{11}+e_{22}+\dots e_{n-1,n-1})A \sim_1 A(e_{11}+e_{22}+\dots e_{n-1,n-1})=e_{12}+e_{23}+\dots+e_{n-2,n-1}$.   Continuing this process, we conclude $A\in \{0\}_{n-1}$. We claim that $A\notin \{0\}_{n-k}$ for $1<k<n$. Let $A\in \{0\}_{n-k}$. Thus $A\sim_{n-k} 0$. By Proposition \ref{First basic prpoerties of connected elements} (3), there exist $X, Y \in R$ such that $A^{(n-k)+l}=X(0)^{l}Y$ ($\forall l\in \mathbb{N}$). If we put $l=k+1$, then $A^{n-1}=0$, which is a contradiction. Hence $d(0, A)=n-1$ for every $n \in \mathbb{N}$. Therefore $diam(R)=\infty$.
\end{proof}
%------------------------

%Since the ring $R=\frac{K\langle X,Y\rangle}{XY-1}$ is obviously not Dedekind finite, we get that $diam(R)=\infty$ and the above proposition \ref{finite quotient} leads to the following corollary.

%\begin{corollary}
%	Let $K$ be a field, the free algebra $K\langle X,Y\rangle$ has infinite diameter.
%\end{corollary}
%%%%%%%%%%%%%%%%
%------------------------------

\section{Constructing the closure of a subset of $R$}

Let us first remark that the commutative closure $\overline{S}$ of a subset
$S \subseteq R$ is the union of the commutative closure $\overline{\{s\}}$ of its elements $s\in S$.   On the other hand to construct the commutative closure of an element we need to use factorizations of this element.
We will use the following tools to build elements of $\overline{s}$, for $s\in R$.
\begin{itemize}
	\item For any $a\in R$, we have $a\sim_1 a(1+b)$ (resp. $a\sim_1 (1+b)a$) for any $b\in l(a)$ (resp. $b\in r(a)$).
	\item Somewhat more general than the above point, let us remark that if $xb=0$ we always have $xy \sim_1(y+b)x$ (resp. if $c, y\in R$ are such that $cy=0$ then $xy\sim_1 y(x+c)$).  We even have, for $a=xy^n$, that $(y+r(x))^nx\in \{a\}_n$ (and for $b=x^ny$ we have $y(x+l(y))^n\in \{b\}_n$).
	\item For $a,b\in R$ we have $a(1+ba)\sim_1 a(1+ab)$.
\end{itemize}

We leave the short proofs of these statements to the reader and start to apply them to different cases.
%----------------------
\begin{example}
	{\rm 
Let $R=M_{2}(k)$ and $A, B\in R$. Also let $A=\left(
\begin{array}{cc}
1 & 0  \\
0 & 0
\end{array} \right)$ and $B=\left(
\begin{array}{cc}
0 & 0  \\
1 & 1 
\end{array} \right)$. Thus $AB=0$ and $BA=\left(
\begin{array}{cc}
0 & 0  \\
1 & 0
\end{array} \right)$. Since $1+AB=I$ and $1+BA=\left(
\begin{array}{cc}
1 & 0  \\
1 & 1 
\end{array} \right)$, so $1+BA\notin \overline{\{1+AB\}}=\overline{\{I\}}=\{I\}$, as wanted.
}
\end{example}
%-------------------------

\begin{lemma}
	\label{class of a conjugate}
	Let $R$ be any ring, and let $a, b\in R$ and $u\in U(R)$ such that $b=uau^{-1}$. Then, for any $n\ge 1$, $\{a\}_n=\{b\}_n$.
\end{lemma}
\begin{proof}
	Let $c$ be any element in $\{b\}_n$. According to Proposition 
	\ref{First basic prpoerties of connected elements}, there are 
	two sequence of elements $x_1,\dots,x_n,y_1,\dots,y_n$ in $r$ 
	such that $b=x_1y_1, \; y_1x_1=x_2y_2,\dots, y_nx_n=c$.
	we thus have $a=u^{-1}bu=(u^{-1}x_1)(y_1u)\sim_1y_1x_1\sim_{n-1}c$.  We thus conclude that $a\sim_n c$.  This yields the result.
\end{proof}
%-----------------------------------
\begin{example}
	{\rm 
	Let $R= k\langle x,y,z,t \rangle /I$, where $k$ is a field and $I$ is the ideal generated by the element $xy-zt$. Then $yx,tz\in \{xy\}_1$ and  $d(yx,tz)=2$.
}
\end{example}

%----------------------------
\begin{theorem}
	\label{}
	Let $k$ be a field and $R=K\left\langle x, y\right\rangle $ be the free $k$-algebra. Then $diam(R)$ is infinite.
\end{theorem}
\begin{proof}
%	Consider $x, z \in R $ such that $zx\ne xz$.  We first show by induction on $l$ that for every $l \in \mathbb{N}$, $d(x+zx^{l} , x+x^{l}z)\le l$.
%	
%	For $l=1$ we have $x+zx=(1+z)x\sim_1 x(1+z)=x+xz$ so ok for $l=1$ 
%	Suppose that for any $z\in R$ such that $zx\ne xz$ we have that
%	$d(x+zx^{l-1}, x+ z^{l-1}x)=l-1$.  In particular, we have 
%	$d(x+(xz)x^{l-1}, x+ x^lz)=l-1$ and since $x+zx^{l}\sim_1
%	x+xzx^{l-1}\sim_{l-1} x+x^lz$, we get $d(x+zx^{l} , x+x^{l}z)\le l$.  In particular we have $d(x+yx^{l}, x+x^{l}y)\le l$.  
%	
%	Let us give an explicit path from $x(1+yx^n)$ and $x(1+x^ny)$
%	we have 
%	
%	 \begin{falign}
%	 $x(1+yx^n)=(1+xyx^{n-1})x\sim_1  x(1+xyx^{n-1})=\\(1+x^2yx^{n-2})x
%	 \sim_1 x(1+x^2yx^{n-2})
%	 =(1+x^3yx^{n-3})x\sim_1 x(1+x^3yx^{n-3})\\
%	 =\dots =(1+x^ny)x\sim_1	x(1+x^ny)
%	$
%	\end{falign}

	Consider $x, y \in R $. We show that for every $l \in \mathbb{N}$, $d(x+yx^{l} , x+x^{l}y)=l$. Since $x+yx^{l}=(1+yx^{l-1})x \sim_{1} x(1+yx^{l-1})=(1+xyx^{l-2})x \sim_{1} x(1+xyx^{l-2})=(1+x^{2}yx^{l-3})x\sim_{1} \dots \sim_{1} x(1+x^{l-1}y)=x+x^{l}y$, so $d(x+xyx^{l} , x(1+x^{l}y))\leq l$. 
	
	The path we just described between the two elements $x+xyx^{l}$ and $x+x^{l}y$ is the only path between these two elements.  This is a consequence of the fact that, for $r,s\ge 1$, the only factorization in $R$ of $x+x^syx^r$ are  $x+x^syx^r=x(1+x^{s-1}yx^r)=(1+x^syx^{r-1})x$.  We leave the arguments to the reader.
\end{proof}

%	We claim that the distance between $x+xyx^{l}$ and $x(1+x^{l}y)$ is exactly $l$. For this, let $x+xyx^{l}=fg$ be another decomposition of $x+xyx^{l}$. We have two following cases:
%\par 
%{\bf Case 1:} Assume that $f=f_{0}+x$ and $g=g_{0}+yx^{3}$. Thus $fg=(f_{0}+x)(g_{0}+yx^{3})=f_{0}g_{0}+f_{0}yx^{3}+xg_{0}+xyx^{3}=x+xyx^{3}$, so $f_{0}(g_{0}+yx^{3})+xg_{0}=x$.
%If we put $x=f_{0}$, then $f_{0}(g_{0}+y(-f_{0})^{3})=-f_{0}(-g_{0}+1)$ and so $-yf_{0}^{3}=f_{0}$. Hence either $f_{0}=0$ or $f_{0}\neq 0$. If $f_{0}\neq 0$, it is impossible. Therefore $f_{0}=0$, as desired.
%\par 
%{\bf Case 2:} Assume that $f=f_{0}+xy$ and $g=g_{0}+x^{3}$. Thus $fg=(f_{0}+xy)(g_{0}+x^{3})=f_{0}g{0}+xyg_{0}+f_{0}x^{3}+xyx^{3}=x+xyx^{3}$. If we put $x=0$, then $f_{0}g_{0}=0$. So $f_{0}=0$ or $g_{0}=0$, which both of them is impossible.
%\\ Hence we can conclude $d(x+xyx^{l} , x(1+x^{l}y))=l$ for any $l\in \mathbb{N}$. Therefore $diam(R)$ is infinite.

%-----------------------
\begin{example}
{\rm
Let $R$ be a ring and $I$ a commutatively closed ideal of $R$. Also, let $diam(\dfrac{R}{I})$ is finite. Then the diameter $diam(R)$ is not necessarily finite. Consider $R= k[x_{1}, x_{2}, \dots, x_{m}, \dots][x; \sigma]$, where $k$ is field and $\sigma$ is an automorphism on $k[x_{1}, x_{2}, \dots, x_{m}, \dots]$ such that $\sigma(x_{i})=x_{i+1}$.
 Since the chain $x_{1}x\sim_{1} xx_{1}=x_{2}x\sim_{1} xx_{2}=x_{3}x\sim_{1} \dots $ is infinite we get that $diam(R)$ is infinite.
 Also assume that $I=(x)$. We can easily see that $I$ is commutatively closed. Since $\dfrac{R}{I}\cong k[x_{1}, x_{2}, \dots, x_{m}, \dots]$ is commutative, we have $diam(k[x_{1}, x_{2}, \dots, x_{m}, \dots])=0$.
}
\end{example} 
%--------------------------
\begin{example} 
{\rm
Let $R$ and $S$ be two rings and $\varphi :R \rightarrow S$ is a morphism.  In general we can get any relation between $diam(R)$ and $diam(S)$.  For instance, there is homomorphism $\varphi : k\left\langle x, y\right\rangle  \rightarrow k$, where $k$ is a field. We know $diam(k)=1$ while $diam(k\left\langle x, y\right\rangle )$ is infinite.
}
\end{example}
%---------------------------------
The next proposition establishes a nice connection of our study with the notion of 
stably associated elements in a ring.  In fact, we will just use a very special case of this notion and hence we do not introduce a formal definition (see PM Cohn \cite{C} for more information about that).  Let us recall that two square matrices $A,B\in M_n(R)$ are said to be {\it associated} if there exist invertible matrices $P,Q$ such that 
$PAQ=B$.
We shall say that $A$ and $B$ are {\it stably associated} if $diag(A,I)$ is associated to $diag(B,J)$ for some unit matrices $I,J$ (not necesserily of same size).
%-------------------------
\begin{proposition}
	\label{stably associated diag matrices}
Let $x,y$ be elements in a ring $R$.  The $2\times2$ diagonal matrices $diag(1+xy,1)$ and $diag(1+yx,1)$ are associated.
\end{proposition}
\begin{proof}
First remark that 
$$
\begin{pmatrix}
-y & -1 \\ 
1 & 0
\end{pmatrix} \begin{pmatrix}
1 & x \\ 
0 & 1
\end{pmatrix} =\begin{pmatrix}
-y & -yx-1 \\ 
1 & x
\end{pmatrix} ;
%{\rm and}  
\begin{pmatrix}
1	& 0 \\ 
-y & 1
\end{pmatrix} \begin{pmatrix}
x & 1 \\ 
-1 & 0
\end{pmatrix} =\begin{pmatrix}
x & 1 \\ 
-yx-1 & -y
\end{pmatrix},
$$
so that the matrices on the right-hand side of the above equalities are invertible.	
We have:
$$
\begin{pmatrix}
-y  & -yx-1 \\ 
1 & x
\end{pmatrix}\begin{pmatrix}
1+xy & 0 \\ 
0 & 1
\end{pmatrix}  \begin{pmatrix}
x & 1 \\ 
-yx-1 & -y
\end{pmatrix} =\begin{pmatrix}
1+yx & 0 \\ 
0 & 1
\end{pmatrix}.
$$ 	
This proves our proposition.
\end{proof}
%--------------------
\begin{remark}
{\rm 
	Let us remark that the above result is due to the fact that, in the language used in 
	\cite{C}, the equality $(1+xy)x=x(1+yx)$ is comaximal.
}
\end{remark}

This leads to the following statement:
%----------------------------
\begin{proposition}
\label{CC and stably associated}
	Let $S$ be a connected subset of a ring $R$.   Then $S$ is commutatively closed if and only if for any two elements $a,b\in S$  we have that the diagonal  matrices $diag(1-a,1)$ and $diag(1-b,1)$ are associated in $M_2(R)$.
\end{proposition}
\begin{proof}
		if $S$ is commutatively closed and connected (equivalently $S=C(a)$ for some $a\in S$) subset of $R$ and $a,b\in S$. then there is a path from $a$ to $b$ in $S$ and, as in Proposition \ref{First basic prpoerties of connected elements}, we have two sequences $x_1,x_2,\dots,x_n$ and $y_1,y_2,\dots, y_n$ such that $a=x_1y_1, y_1x_1=x_2y_2, y_2x_2=x_3y_3,
		\dots , y_nx_n=b$ and Proposition \ref{stably associated diag matrices} easily implies that $diag(1-a,1)$ and $diag(1-b,1)$ are associated.
		
	   Conversely, if $a=xy\in S$, then Proposition \ref{stably associated diag matrices} gives that the diagonal matrices $diag(1-xy,1)$ and  $diag(1-yx,1)$ are associated and hence $yx\in S$. 
\end{proof}

In the next corollary, we get some generalizations of classical examples.

\begin{corollary}
	The set $1-S$ is commutatively closed for any one of the following subsets $S$ of $R$:
	\begin{enumerate}
		\item $S=U_r(R)$ (resp. $S=U_l(R)$ or $S=U(R)$), the set of right (rep. left, two sided) invertible elements of $R$.
		\item  $S=\{ A\in M_n(K)\mid Rank(A)=l\}$, where $K$ is a field and $l\le n\in \mathbb N$.
		\item $S= reg(R)$ the set of regular elements of $R$.
	\item $S$ is the set of strongly $\pi$-regular elements.
		\item $S$ is the set of left (right) zero divisors in $R$.
	\end{enumerate}
\end{corollary}
\begin{proof}
	We refer the reader to \cite{AL} for the proofs or references for these statements.
\end{proof}
 %--------------------
\section{On Commutatively closed graph over matrix rings}
In this section, we study some properties of the commutatively closed graph over matrix rings.

\par For a ring $R$ and $n\in \mathbb{N}$, we denote $N_n(R)$ the set of elements of $R$ that are nilpotent of index $n$. 

\begin{proposition}
	\label{0_i is contained in N_i+1}
	Let $R$ be a ring.  Then
	\begin{enumerate}
		\item For any $i\in \mathbb{N}$, we have 
		$\{0\}_i \subseteq N_{i+1}(R)$.  In particular, 
		$\overline{0}\subseteq  N(R)$.
		\item For any strictly upper triangular matrix $U\in M_n(R)$,
		$U\in \{0\}_{n-1} \subseteq \overline{\{0\}}$.
		\item Let $U_n(R)\subseteq M_n(R)$ be the set of all $n\times n$ strictly upper triangular matrix over $R$. Then $diam\big(U_n(R)\big)\le 2(n-1)$.
	\end{enumerate}
\end{proposition}
\begin{proof}
(1)This is easily proved by induction using Proposition \ref{First basic prpoerties of connected elements}.

(2) We may assume 
that $U\ne 0$ and we denote the lines of $U$ by $L_1,L_2,\dots, L_n$.  In fact, 
the last line $L_n$ is zero, and we  define $r\in \{1,\dots,n-1\}$ to be minimal 
such that $L_i$ is zero for $i>r$.   We will prove that $U\in \{0\}_r$ 
by induction on $r$.     
We write 
$$
U=\begin{pmatrix}
I_{r,r} & 0 \\ 
0 & 0
\end{pmatrix}U \quad {\rm and }\quad
B:=U\begin{pmatrix}
I_{r,r} & 0 \\ 
0 & 0
\end{pmatrix}\in \{U\}_1, 
$$
where $I_{r,r}$ denotes the identity matrix of size $r\times r$.

\noindent If $r=1$, we get that $B=0\in M_n(D)$ and this yields the thesis.

\noindent If $r>1$, write $B=(R_1,\dots, R_n)$ where $R_i$ is the $i^{th}$ row of $B$.   The matrix $B$ is easily seen to be upper triangular and such that 
the rows $R_r,\dots, R_n$ are zero.   This means that this matrix has at least one more zero row than the matrix $U$.  The induction hypothesis gives that 
$B\in \{0\}_{r-1}$, but then $U\in \{B\}_1\subseteq \{0\}_r\subseteq \overline{\{0\}} $, as required.

\noindent (3) By the above statement (2), we know $U_n(R)\subseteq \{0\}_{n-1} \subseteq \overline{\{0\}}$.  So that for two matrices $A,B\in U_n(R)$, we have $A\sim_{n-1} 0 \sim_{n-1} B$.  This yields the conclusion.  
\end{proof}

%???In the case when $R$ is a division ring, we will have az sharper bound on the diameter (see ***°).

\smallskip

We will now determine the diameter of the class $C(0)\in M_n(D)$ where 
$D$ is a division ring.  
%----------------------------------
The following lemma is well known, we give a proof for completeness.
\begin{lemma}
	\label{Nilpotent matrices division ring and upper triangular}
	Every nilpotent matrix with coefficients in a division ring is similar to a strictly upper triangular matrix.
\end{lemma}
\begin{proof}
	The proof is based on the fact that any nonzero column
	can be the first column of an invertible matrix.  
	So let $A\in M_n(D)$ be a nilpotent matrix with coefficients in a division ring $D$ and let $u\in M_{n,1}(D)$ be a nonzero column such that $Au=0$.
	Let $U\in M_n(D)$ be an invertible matrix having $u$ as its first column.  We conclude that 
	$$U^{-1}AU=\begin{pmatrix}
	0 & r \\ 
	0 & A_1
	\end{pmatrix},   
	$$ 
	for some row $r\in M_{1,n-1}(D)$.  It is easy to check that $A_1$ is again nilpotent. An easy induction on the size of the nilpotent matrix yields the proof.
	\end{proof}

%****comments the preceding thm is true over an Hermite ring ,,,
%we would have to show that there exists a unimodular column in the kenel of the matrix.****

The next result was proved in \cite{AL} for matrices with coefficients over fields.  
%-------------------------
\begin{proposition}
	\label{The class of Zero in matrix rings}
	Let $D$ be a division ring and $n\in \mathbb N
	$, the class of $\overline{\{0\}}$ in $R=M_n(D)$ is 
	the set of nilpotent matrices. 
\end{proposition}
\begin{proof}
	We have seen that, in any ring, 
	$\{0\}_i\subseteq N(R)_{i+1}$ (cf.\ Proposition \ref{0_i is contained in N_i+1}).  
	Conversely, if $A\in M_n(D)$ is nilpotent, the above lemma \ref{Nilpotent matrices division ring and upper triangular} shows that there exists an invertible matrix $P$  
	and a strictly upper triangular matrix $U\in M_n(D)$ such that $PAP^{-1}=U$.  
	Since the class of an element is  the same as the class of any of its 
	conjugate, we conclude that
	$\overline{\{A\}}=\overline{\{U\}}$.  Proposition \ref{0_i is contained in N_i+1} (1) and (2) implies that 
	$\overline{\{U\}}=\overline{\{0\}}$.   
%	Since $U$ 
%	is nilpotent, thanks to Proposition \ref{0_i is contained in N_i+1}, we only need to prove that $U\in \overline{\{0\}}$.  We may assume 
%	that $U\ne 0$ and we denote the lines of $U$ by $L_1,L_2,\dots, L_n$.  In fact, 
%	the last line $L_n$ is zero, and we  define $r\in \{1,\dots,n-1\}$ to be minimal 
%	such that $L_i$ is zero for $i>r$.   We will prove that $U\in \overline{\{0\}}$ 
%	by induction on $r$.     
%	We write 
%	$$
%	U=\begin{pmatrix}
%	I_{r,r} & 0 \\ 
%	0 & 0
%	\end{pmatrix}U \quad {\rm and }\quad
%	B:=U\begin{pmatrix}
%	I_{r,r} & 0 \\ 
%	0 & 0
%	\end{pmatrix}\in \{U\}_1, 
%	$$
%	where $I_{r,r}$ denotes the identity matrix of size $r\times r$.
%	
%	\noindent If $r=1$, we get that $B=0\in M_n(D)$ and this yields the thesis.
%	
%	\noindent If $r>1$, write $B=(R_1,\dots, R_n)$ where $R_i$ is the $i^{th}$ row of $B$.   The matrix $B$ is easily seen to be upper triangular and such that 
%	the rows $R_r,\dots, R_n$ are zero.   This means that this matrix has at least one more zero row than the matrix $U$.  The induction hypothesis gives that 
%	$B\in \overline{\{0\}}$, but then $U\in \{B\}_1\subseteq \overline{\{0\}}$, as required.
%	Now, assume that a matrix $A\in N(R)_{i+1}$, then by what we just proved $A\in C(0)$ and hence there exists
%	$l\in \NN$ such that $A\in \{0\}_l\subseteq N(R)_{l+1}$, where the last inclusion is due to Proposition \ref{0_i is contained in N_i+1}.   Thus we conclude that $l\ge i$.
\end{proof}
%---------------------------
%Maybe thi result is also true for matrices over Hermite rings

A Jordan block $J_l$ (associated to zero) is a matrix of the form
$$
J_l=\begin{pmatrix}
0 & 1 & 0 & 0 & \dots & 0 \\ 
0 & 0 & 1 & 0 & \dots & 0 \\ 
0 & 0 & 0 & 1 & \dots & 0 \\ 
\dots &  \dots &  \dots &  \dots &  \dots &   \dots \\ 
0 & 0 & 0 & \dots & \dots & 1 \\ 
0 & 0 & 0 & 0 & \vdots & 0
\end{pmatrix}\in M_l(D) \quad \quad (1)
$$

If $A\in M_n(D)$ is a nilpotent matrix, where $D$ is a division ring,  $A$ is similar to a diagonal sum of Jordan blocks.  This is classical if $D$ is commutative and for a proof in a noncommutative setting we may refer to Chapter $8$ of P.M. Cohn's book (\cite{C}) or to the more recent paper \cite{HB}.  Let us notice that $J_1=0$. 
%------------------
\begin{lemma}
	\label{Class of a Jordan block}
	Let $J_l$ be a matrix block of size $l>1$. Then $J_l\sim_{l-1} 0$ 
\end{lemma}
\begin{proof}
	As above, we write $J_l$ for the Jordan matrix presented in (1).
	
	We proceed by induction on $l$.  If $l=2$, we have 
	$$
	J_2=\begin{pmatrix}
	0 & 1 \\ 
	0 & 0
	\end{pmatrix} = \begin{pmatrix}
	1 & 0 \\ 
	0 & 0
	\end{pmatrix}\begin{pmatrix}
	0 & 1 \\ 
	0 & 0
	\end{pmatrix}
\sim_1 \begin{pmatrix}
0 & 1 \\ 
0 & 0
\end{pmatrix}\begin{pmatrix}
1 & 0 \\ 
0 & 0
\end{pmatrix}=\begin{pmatrix}
0 & 0 \\ 
0 & 0
\end{pmatrix}.$$

Similarly, for $l>2$, we have
$$
J_{l}=(\sum_{i=1}^{l-1} e_{ii})J_{l}\sim_1 J_{l}(\sum_{i=1}^{l-1} e_{ii})
=\begin{pmatrix}
J_{l-1} & 0 \\ 
0 & 0
\end{pmatrix}\sim_{l-2} 0. $$ 
Where we have used the induction hypothesis: $J_{l-1}\sim_{l-2} 0$.
This implies that $J_{l}\sim_{l-1} 0$, as required.
\end{proof}
%---------------------
Let us extract from the above proof the following observation.

%rom the proof of the above lemma, we also get
%the following result giving an upper bound for the distance between two nilpotet matrices.
\begin{corollary}
	\label{J_L+1 sim_1 (J_l,J1)}
	If $l>1$, we have $J_{l}\sim_1 diag(J_{l-1}, J_1)$.
\end{corollary}
%-----------------------------
We will now look more closely to the class of nilpotent matrices.
\begin{proposition}
	\label{Nilpotent matrices}
	Let $D$ be a division ring. Then for $A\in M_n(D)$, we have 
	$$
	A^{l+1}=0 \Leftrightarrow A\sim_l 0.
	$$
\end{proposition}
\begin{proof}
	We suppose that $A\ne 0$.
	We know (cf. Proposition \ref{The class of Zero in matrix rings}) that the class $C(0)$ consists of all the 
	nilpotent matrices.  Thanks to Lemma \ref{class of a conjugate} we know that we can replace $A$ by a conjugate to evaluate the length of path from $A$ to $0$.  
	Using  a result from \cite{HB}, we know that there 
	exists an invertible matrix $P\in GL_n(D)$ such that $PAP^{-1}$ is of the form $diag(J_{n_1},J_{n_2},\dots,J_{n_s})$ where the square 
	matrices $J_i\in M_i(D)$ are of the form given in (1) above.  
	Moreover the maximal size of the Jordan blocks is $l+1$ i.e., for all 
	$1\le i \le s$, we have $n_i\le l+1$.  Since for any integer $i\in 
	\mathbb N$, we have $\{0\}_i \subseteq \{0\}_{i+1}$, the above Lemma 
	\ref{Class of a Jordan block} implies that, for all $1\le i \le s$,
	$J_{n_i}\sim_l 0$.   This easily leads to the conclusion that $A\sim_l 0$. 
	
	\noindent The converse was proved in Lemma \ref{0_i is contained in N_i+1}.
\end{proof}
%------------------------------------
\begin{theorem}
\label{Diam of the class of nilpotent matrices}
Let $R=M_{n}(D)$, where $D$ is a division ring. Then, for nilpotent matrices $A,B\in M_n(k)$, with nilpotent indexes $n(A), n(B)$ respectively.  We have $d(A,B)\le max\{n(A),n(B)\} -1$.  In particular,  $diam_{R}\big(C(0)\big)=n-1$.
\end{theorem}
\begin{proof}
	We know that the class of zero in $M_n(D)$ is exactly the set of nilpotent matrices.  Hence the matrices $A,B\in C(0)$ and the distance between $A$ to $B$ is the shortest path from $A,B$ in the graph defined by $C(0)$.  Let us write $l=n(A)$ and $s=n(B)$, by symmetry we may assume that $l\ge s$.  Let $diag(J_{n_1},J_{n_2},
	\dots, J_{n_r})$ be the Jordan form of $B$, where $r\ge 1$ and $s=n_1\ge n_2,\dots \ge n_r$.
	We will use induction on $r$.  In the proof, to avoid heavy notations, we will write, $(c_1,c_2,\dots ,c_t)$ for $diag(c_1,c_2,\dots,c_t$) (where  $c_1,c_2,\dots,c_t $ are square matrices).
	
	\noindent If $r=1$, then $B=J_s$ and $A=(J_l,A')$ where $A'$ is a nilpotent matrix of index $\le l$. Using repeatedly Corollary 
	\ref{J_L+1 sim_1 (J_l,J1)}, we can write $A=(J_l,A')\sim_{l-s} (J_s,A'')$ and $(A'')^s=0$, so $d(A,(J_s,A''))\le l-s$ and
	$$
	d(A,B)\le d(A,(J_s,A''))+ d((J_s,A''),B).
	$$
	Using Proposition \ref{Nilpotent matrices}, we get $d((J_s,A''),B)=d((J_s,A''),J_s)=d(A'',0)\le s-1$ and we conclude 
	$d(A,B)\le l-1$, as required.
	
	Suppose now that the formula is proved for matrices $B$ having less than 
	$r> 1$ Jordan blocks and consider a matrix 
	$B=(J_{n_1},\dots,J_{n_r})=(J_s,B')$, with $(B')^s=0$.   As above we have $A=(J_l,A')\sim_{l-s}(J_s,A'')$ and the 
	induction hypothesis gives also $d(A'',B')\le max\{n(A''),n(B')\}-1$.   
This gives $d(A,B)\le d(A,(J_s,A''))+d((J_s,A''),B)\le l-s + d(A'',B')\le 
l-s + max\{n(A''),n(B')\}-1\le l-s+s-1=l-1$, as required.
	
	In particular, since the maximal index of nilpotency for matrices in $R=M_n(k)$ is $n$, and $d(J_n, 0)=n-1$ we get that $diam_RC(0)=n-1$. 
\end{proof}

We will now show that the diameter of the matrix ring $M_n(D)$, over a division ring $D$, is itself $n-1$.
The following lemma is far from surprising but needs to be proved.

\begin{lemma}
	\label{Size of diagonal}
	Suppose that $D$ is a division ring and that $A,B\in M_n(D)$ are of the form
	$$
	A=\begin{pmatrix}
	U & 0 \\ 
	0 & N
	\end{pmatrix}, B=\begin{pmatrix}
	V & 0 \\ 
	0 & M
	\end{pmatrix}  
	$$
	where $U\in GL_r(D)$ and $V\in Gl_s(D)$ are invertible matrices and $N,M$ are nilpotent matrices.
	If $A\sim B$ then $r=s$.
\end{lemma}
\begin{proof}
	Let us first remark that there exists $l\in \NN$ such that $N^l=µ0$ and $M^l=0$.  Theorem \ref{The graph C(a)} implies that we may assume $M=0$ and $N=0$.
	Now, Taking powers again, Proposition \ref{First basic prpoerties of connected elements} shows that we may assume  $A\sim_1 B$.   So let us write $A=XY, B=YX \in M_n(R)$ and decompose $X$ and $Y$ as follows  
	$$
	X=\begin{pmatrix}
	X_1 & X_2 \\ 
	X_3 & X_4
	\end{pmatrix}, \quad {\rm and } \quad Y=\begin{pmatrix}
	Y_1 & Y_2 \\ 
	Y_3 & Y_4
	\end{pmatrix} 
	$$ 
	where $X_1\in M_{r\times s}(D)$, $Y_1\in M_{s\times r}(D)$.  This fixes the size of all the other matrices appearing in $X$ and $Y$.    Since $M$ and $N$ are zero, the two equations $AX=XB$ and $YB=BY$ quickly imply that $X_2,X_3,Y_2,Y_3$ are all zero. and we get $U=X_1Y_1$ and $V=Y_1X_1$.  Since $U$ and $V$ are invertible we easily conclude that $r=s$.  
\end{proof}

\begin{theorem}
	\label{Diam of the matrix ring}
	Let $D$ be a division ring $n\in \NN$.  Then 
	$$
	diam(M_n(D))=n-1.
	$$ 
\end{theorem} 
\begin{proof}
	A consequence of the Fitting Lemma is that any matrix $A\in M_n(D)$ is similar to a block diagonal matrix of the form $diag(U,N)$ where $U$ is an invertible matrix and $N$ a nilpotent matrix. 
	%****Maybe more should be said here*****.  
	%In fact you decompose $D^n=ker f^l) \oplus Im(f^l) $ and you take a basis for D^n which is adapted to this decomposition then f restricted to $ker(f^l)$ is nilpotent and $f$ restricted to the %im(f^l) is an isomorphism. 
	  Lemma \ref{class of a conjugate} shows that to compute the distance between two different matrices that 
	  are in the same commutative class, we may use similar matrices.  Thus we need to compute $d(A,B)$ where $A$ and $B$ are of the form
	  $A=diag(U,N)$ and $B=diag(V,M)$. 
	The preceding lemma \ref{Size of diagonal} shows that $U, V\in Gl_s(D)$ and $N,M \in M_{n-s}(D)$.  Theorem \ref{Diam of the class of nilpotent matrices} and  Lemma \ref{6} show that we may assume $0<s<n$.
	  %Since we already know the distance between two nilpotent matrices, we may assume that the matrices $U$ and $V$ are present.  Moreover, since the distance between two invertible matrices that are in the same class is $\le 1$, we may assume that both $A$ and $B$ are not invertible.  
%	Since  the matrices $A$ ad $B$ are in the same class, their 
%	characteristic polynomials are the same.  Noting $\chi_A$ 
%	the characteristic polynomial of $A$, we thus have $\chi_U\chi_N=\chi_V\chi_M$.   Since the matrices $U$ and $V$ are invertible, $\chi_U$ and $\chi_V$ are polynomials with nonzero independant terms.  Since $M$ and $N$ are nilpotent matrices we know that $\chi_N=X^s$ and $\chi_M=X^l$ where $s$ and $l$ are equal to the size of the $N$ and $M$ respectively.  From all this, we easily get that $1\le s=l<n$.   So the diagonal matrices appearing in the matrices $A$ and $B$ are of the same size.     

Since the matrices $N$ and $M$ are nilpotent we conclude that their distance is less or equal to $n-2$. This means that there is a sequence of factorizations of length $\le n-2$ linking $M$ and $N$.   We claim that the matrices $U$ and $V$ are in fact similar.  Assume that $A\sim_r B$, according to  Proposition \ref{First basic prpoerties of connected elements} we know that there exist matrices $X,Y\in M_n(D)$ such that $AX=XB$ and $YA=BY$, moreover for any $l\in \NN$, we have that $A^{r+l}=XB^lY$.  Choosing $l$ such that $M^l=N^l=0$, and writing $X, Y$ as blocks matrices with $X_1,Y_1 \in M_s(D)$, this last equality shows that 
	$$
	\begin{pmatrix}
	U^{r+l} & 0 \\ 
	0 & 0
	\end{pmatrix} =\begin{pmatrix}
	X_1 & X_2 \\ 
	X_3 & X_4
	\end{pmatrix} \begin{pmatrix}
	V^{l} & 0 \\ 
	0 & 0
	\end{pmatrix} \begin{pmatrix}
	Y_1 & Y_2 \\ 
	Y_3 & Y_4
	\end{pmatrix}. 
	$$   
    Comparing the blocks on the top left corner gives $U^{r+l}=X_1V^lY_1$.  Since $U$ is invertible, we conclude that the matrices $X_1$ and $Y_1$ are also invertible.  Now, comparing again the top left corner blocks of the equality $AX=XB$, we get $UX_1=X_1V$.  This shows that $U$and $V$ are similar, as claimed.  This implies that $U\sim_1 V$.   Since $A=diag(U,N)$ and $B=diag(V,M)$ (with sizes of $U$ and $N$ greater or equal to $1$).  We conclude, thanks to Theorem \ref{Diam of the class of nilpotent matrices}, that $d(A,B)=d(M,N)\le n-2$.   Hence the class with the biggest distance is the class of nilpotent matrices, so that Theorem \ref{Diam of the class of nilpotent matrices} implies that $Diam (M_n(D))=n-1$.   
\end{proof}
%-----------------------------
\begin{theorem}
	Let $R$ be a semisimple ring and $R= M_{n_1}(D_1)\times \dots \times M_{n_l}(D_l)$ be its Wedderrburn Artin decomposition where $D_1,\dots, D_l$ are division rings.  Then 
	$$
	diam(R)={\rm max}\{n_i-1 \mid 1\le i \le l\}.
	$$
\end{theorem}
\begin{proof}
	This is a simple consequence of Theorem \ref{Diam of the matrix ring} and Proposition \ref{diameter of a product}
\end{proof}

%%%%%%%%%%%%%%%%%%%%%%%%%
\begin{definition}
The girth of a graph $G$, denoted by $gr(G)$, is the length of shortest cycle in $G$, provided $G$ contains a cycle; otherwise $gr(G)=\infty$.
\end{definition}
%------------------------
Note that if $R$ is a ring, then we define the commutatively closed girth of $R$ as follows:
\begin{center} $gr(C(R))=min\{ gr\big(C(a)\big) ~|~ a\in R \}$.\end{center}
%--------------------------
\begin{theorem}
\label{3.1.1}
Let $D$ be a division ring. Then $gr_{M_{n}(D)}\big(C(0)\big)=3$.
\end{theorem}
\begin{proof}
 We know the class of $\overline{\{0\}}$ in $R=M_n(D)$ is 
	the set of nilpotent matrices, where $D$ is a division ring (cf. Proposition \ref{The class of Zero in matrix rings}). 
 It is enough, to find three nilpotent matrices that form a cycle. It is easy to see that  $E_{1n}, E_{n1} \in \{0\}_1$ (Since $E_{1n}=E_{11} (E_{1n})$, $0=E_{1n}(E_{11})$, $E_{1n}=E_{1n}E_{nn}$ and $E_{nn}E_{1n}=0$) and $E_{1 n}\sim_1 E_{n1}$. Hence $E_{1,n}$, $E_{2,n-1}$ and $0$ form a cycle. 
Therefore $gr\big(C(0)\big)=3$.
\end{proof}
\par As an immediate consequence of Theorem \ref{3.1.1} and definition of $gr(C(R))$ , we get the following.
\begin{corollary}
Let $D$ be a  division ring. Then $gr\big(C(M_{n}(D))\big)=3$.
\end{corollary}
\begin{corollary}
	if $R$ is a semisimple ring then $gr\big(C(R)\big)=3$.
\end{corollary}

\end{document}